\documentclass[letterpaper,10pt,conference]{ieeeconf}

\IEEEoverridecommandlockouts                              
\title{\LARGE \bf
Improved Rates for Stochastic Variance-Reduced \\ Difference-of-Convex Algorithms}

\author{Anh Duc Nguyen \; Alp Yurtsever \; Suvrit Sra \; Kim-Chuan Toh
\thanks{AD. Nguyen is with the Institute of Operations Research and Analytics, National University of Singapore, Singapore ({\tt\small ducna@nus.edu.sg}).}%
\thanks{KC. Toh is with the Department of Mathematics, National University of Singapore, Singapore ({\tt\small mattohkc@nus.edu.sg}).}%
\thanks{A. Yurtsever is with the Department of Mathematics and Mathematical Statistics, Umeå University, Sweden ({\tt\small alp.yurtsever@umu.se}).}%
\thanks{S. Sra is with the School of Computation, Information and Technology, Technical University of Munich, Germany ({\tt\small s.sra@tum.de}).}%
} 

\usepackage{changes}

\usepackage{booktabs} 
\usepackage{amsmath}
\usepackage{amsfonts} 

\usepackage{mathtools, algorithmic}

\usepackage[linesnumbered,ruled]{algorithm2e}
\usepackage{color}

\usepackage{amsthm}
\usepackage{amsfonts}
\usepackage{amsmath}
\usepackage{amssymb,bbm}
\usepackage{caption}
\usepackage{hyperref}

\long\def\comment#1{}
\comment{
\setlength{\topmargin}{0 in}
\setlength{\textwidth}{5.5 in}
\setlength{\textheight}{8 in}
\setlength{\oddsidemargin}{0.5 in}
}

\def\argmin{\mathrm{arg} \min}

\newcommand{\inner}[2]{\langle#1, #2\rangle}

\graphicspath{{./image/}}

\newcommand{\ba}{\begin{array}}
\newcommand{\ea}{\end{array}}

\newcommand{\defeq}{\vcentcolon=}

\definecolor{light-gray}{gray}{0.9}
\definecolor{darkblue}{rgb}{0.0,0.0,0.65}
\definecolor{darkred}{rgb}{0.2,0.0,0.0}
\hypersetup{
	colorlinks = true,
	citecolor  = darkblue,
	linkcolor  = darkred,
	filecolor  = darkblue,
	urlcolor   = darkblue,
}

\theoremstyle{plain}
\newtheorem{theorem}{Theorem}

\newtheorem{lemma}[theorem]{Lemma}
\newtheorem{corollary}[theorem]{Corollary}
\theoremstyle{definition}
\newtheorem{definition}[theorem]{Definition}
\newtheorem{assumption}[theorem]{Assumption}
\theoremstyle{remark}
\newtheorem{remark}[theorem]{Remark}

\begin{document}

\maketitle
\thispagestyle{empty}
\pagestyle{empty}

\begin{abstract}
    In this work, we propose and analyze DCA-PAGE, a novel algorithm that integrates the difference-of-convex algorithm (DCA) with the ProbAbilistic Gradient Estimator (PAGE) to solve structured nonsmooth difference-of-convex programs. In the finite-sum setting, our method achieves a gradient computation complexity of \(\mathcal{O}(N + N^{1/2}\varepsilon^{-2})\) with sample size $N$, surpassing the previous best-known complexity of \(\mathcal{O}(N + N^{2/3}\varepsilon^{-2})\) for stochastic variance-reduced (SVR) DCA methods. Furthermore, DCA-PAGE readily extends to online settings with a similar optimal gradient computation complexity \(\mathcal{O}(b + b^{1/2}\varepsilon^{-2})\) with batch size $b$, a significant advantage over existing SVR DCA approaches that only work for the finite-sum setting. We further refine our analysis with a gap function, which enables us to obtain comparable convergence guarantees under milder assumptions.
\end{abstract}

\section{INTRODUCTION}
We study the following nonconvex nonsmooth problem:
\begin{align}
    \label{eq:prob_main}
    \min_{x \in \mathbb{R}^n} F(x)=G(x)-H(x)+r_1(x)-r_2(x),
    \tag{M}
\end{align}
where $G, H:=\frac{1}{N} \sum_{i=1}^N h_i, h_i: \mathbb{R}^n \rightarrow \mathbb{R}$ and $r_1, r_2: \mathbb{R}^n \rightarrow \mathbb{R} \cup\{+\infty\}$ are convex, lower semicontinuous functions, and $h_i$ are differentiable. Such problems arise in various applications, particularly in machine learning, where $G-H$ represents the data-fitting term, i.e. the loss function, while $r_1-r_2$ corresponds to the regularization term \cite{le2022stochastic}. A convex regularizer admits the trivial DC decomposition $r = r-0$. Moreover, popular nonconvex regularizers, like capped $\ell_1$ \cite{zhang2010analysis} and SCAD \cite{fan2001variable}, also possess DC decompositions \cite{ong2013learning, le2008gene}. 

Broadly speaking, the structure presented in \eqref{eq:prob_main} falls within the well‐established class of difference‑of‑convex (DC) programs, a powerful framework for tackling challenging nonconvex nonsmooth optimization problems by decomposing them as the difference of two convex functions \cite{tao1997convex, le2018dc}. A prominent method for DC programs is the difference-of-convex algorithm (DCA) \cite{tao1997convex}, renowned for its simplicity and flexibility \cite{le2024open}. In its simplest form, DCA alternately linearizes the concave component and solves a convex subproblem, guaranteeing monotonic descent and convergence to a critical point under mild assumptions \cite{tao1997convex}.

The subclass of DC problems represented by \eqref{eq:prob_main} encompasses a broad range of applications \cite{le2022stochastic}. One prominent example is regularized empirical risk minimization (ERM), which underpins essential techniques such as LASSO~\cite{Tibshirani1996} and principal component analysis (PCA)~\cite{Hastie2009}. In particular, the regularized ERM has the form
\begin{small}
\begin{align*}
    \min_{x \in \mathbb{R}^n}  \frac{1}{N} \sum_{i=1}^N f_i(x) + r(x),
\end{align*}
\end{small}
 where $f_i$ are $L$-smooth and not necessarily convex for all $i \in [N]$ and $r$ is a regularizer with DC structure, i.e. $r(x) = r_1(x) - r_2(x)$ for convex functions $r_1$ and $r_2$. ERM can be reformulated into \eqref{eq:prob_main} as follows
 \begin{small}
\begin{align}
    \label{eq:prob_spec}
      \min_{x \in \mathbb{R}^n} \frac{L\|x\|^2}{2}  -\frac{1}{N} \sum_{i=1}^N \left(\frac{L\|x\|^2}{2}-f_i(x) \right) + r_1(x) - r_2(x).
    \tag{S}
\end{align}
\end{small}
Note that in this case, $L\|x\|^2/2 -f_i(x)$ is convex since $f_i$ are $L$-smooth for all $i \in [N]$.

Beyond the finite-sum setting, many real-world applications involve large-scale or streaming data, where 
$N$ can be extremely large or even effectively infinite \cite{shalev2012online}. In these scenarios, we also consider the online formulation of $H(x)$ 
\begin{align*}
H(x) \defeq \mathbb{E}_{\zeta \sim \mathcal{D}}[\Phi(x, \zeta)],
\end{align*}
where $\zeta$ is a random data instance and $\Phi$ is the instantaneous sample loss. We use the same notation as in the finite-sum case, writing $h_i(x)\defeq \Phi(x,\zeta_i)$ for i.i.d.\ samples $\{\zeta_i\}$ drawn from the stream, and we treat $N$ as large (or unbounded).

In addition, recent applications of DC programming span various engineering disciplines, including signal recovery \cite{pham2024proximal}, signal processing \cite{le2013sparse}, discrepancy estimation in domain adaptation \cite{awasthi2024best}, neural networks optimization \cite{awasthi2024dc}, shortcut architecture \cite{sun2024understand}, unsupervised learning \cite{thi2007fuzzy}, biomedical signal processing \cite{su2025romp}, and reinforcement learning \cite{le2019unified}. As these applications scale up, the demand for efficient DCA algorithms and their variants continues to grow. Incorporating stochastic methods is a natural extension given their success in convex and smooth nonconvex optimization \cite{lan2020first}.
However, significant challenges remain when adapting stochastic methods like stochastic variance reduction (SVR) to nonsmooth nonconvex settings. Since DCA is particularly well-suited for handling nonsmooth and nonconvex problems, it offers a promising framework for integrating these techniques. While \cite{le2022stochastic} has introduced two SVR DCA methods, there still exists a gap between the optimal gradient computation complexity for smooth nonconvex optimization \(\mathcal{O}(N + N^{1/2}\varepsilon^{-2})\) \cite{li2021page} and that for the nonsmooth case  \(\mathcal{O}(N + N^{2/3}\varepsilon^{-2})\) \cite{le2022stochastic}. Our goal is to bridge this gap and further refine the analysis of SVR DCA methods in more general cases, like online settings.
\subsection{Contributions}
In this work, we introduce DCA-PAGE, a novel SVR DC algorithm that leverages the ProbAbilistic Gradient Estimator (PAGE) \cite{li2021page}. Our key contributions are:
\begin{itemize}
    \item Improved Complexity: Compared to existing stochastic DC algorithms, DCA-PAGE achieves the complexity of \(\mathcal{O}(N + N^{1/2}\varepsilon^{-2})\), improving the SOTA complexity for SVR DCA methods \(\mathcal{O}(N + N^{2/3}\varepsilon^{-2})\)~\cite{le2022stochastic} for problem (\ref{eq:prob_main})\footnote{MM-SARAH \cite{phan2024stochastic} achieves a similar optimal rate for the decomposition \eqref{eq:prob_spec} which is a \textit{special} case of \eqref{eq:prob_main}.}. We compare the complexity with other related DCA methods in Table \ref{tab:stochastic_DCA}. 
    \item Online Settings: As DCA-PAGE only needs occasional large-minibatch updates, we readily extend the convergence analysis to the online case, overcoming a key limitation of previous SVR DCA methods \cite{le2022stochastic}.
    \item Refined Analysis with Gap Function: We propose an alternative measure of optimality for SVR DCA methods. This \textit{gap} measure enables DCA-PAGE to maintain similar convergence guarantees while imposing fewer smoothness assumptions than earlier approaches.
    \item Empirical Validation: We demonstrate the effectiveness of our method in practical applications, specifically in nonconvex binary and multi-class logistic classification.
\end{itemize}
\begin{table}[ht]
\caption{Overview of SVR DCA and Related Methods}
\label{tab:stochastic_DCA}
\begin{center}
\begin{tabular}{|c|c|c|c|}
\hline
Algorithm & Problem & Complexity & Online\\
\hline
DCA-SAGA \cite{le2022stochastic} & \eqref{eq:prob_main} & $\mathcal{O}\left(N + N^{3/4} \varepsilon^{-2} \right)$ & $\times$ \\
DCA-SVRG \cite{le2022stochastic} & \eqref{eq:prob_main} & $\mathcal{O}\left(N + N^{2/3}\varepsilon^{-2}\right)$ & $\times$ \\
MM-SAGA \cite{phan2024stochastic} & \eqref{eq:prob_spec} & $\mathcal{O}\left(N +N^{2/3}\varepsilon^{-2} \right)$ & $\times$ \\
MM-SVRG \cite{phan2024stochastic} & \eqref{eq:prob_spec} & $\mathcal{O}\left(N + N^{2/3}\varepsilon^{-2} \right)$ & $\times$ \\
MM-SARAH \cite{phan2024stochastic} & \eqref{eq:prob_spec} & $\mathcal{O}\left(N + N^{1/2} \varepsilon^{-2} \right)$ & $\times$ \\
DCA-PAGE (Alg \ref{alg:dca_page}) & \eqref{eq:prob_main} & $\mathcal{O} \left(N + N^{1/2}\varepsilon^{-2} \right)$ & $\checkmark$\\
\hline
\end{tabular}
\end{center}
\end{table}
\subsection{Related Work}
Regarding SVR DCA variants, \cite{le2022stochastic} constructs the convex subproblems by employing SVRG \cite{johnson2013accelerating} and SAGA \cite{defazio2014saga} gradient estimators, while \cite{le2017stochastic, le2020stochastic} adopts the SAG estimator. Alternatively, \cite{xu2019stochastic} employs deterministic DCA and utilizes stochastic algorithms such as \cite{johnson2013accelerating} to solve the convex subproblems. More recently, \cite{phan2024stochastic} focuses on the particular DC formulation \eqref{eq:prob_spec} and introduces SVR Majorization-Minimization (MM) algorithms with variance reduction. The authors employ SVRG \cite{johnson2013accelerating}, SAGA \cite{defazio2014saga}, and SARAH \cite{nguyen2022finite} gradient estimators for the construction of the iterate update. Although \cite{phan2024stochastic} does not require the regularizer $r$ to have a DC decomposition (as in \eqref{eq:prob_main}), their study is confined to the \textbf{specialized} setting \eqref{eq:prob_spec}. For online DCA variants, \cite{le2020online, le2022online} show only convergence of their online DCA variants without any computational complexity albeit in very general settings. Regarding the gap function measure, recent works \cite{maskan2024block, maskan2025revisiting} show its effectiveness for DCA-type algorithms in some structured constrained settings. 

\subsection{Notations}
We use $[n]$ for the set $\{1, \ldots, n\}$. Let $\|\cdot\|$ denote the Euclidean norm for a vector. $\langle \cdot, \cdot\rangle$ denotes the inner product. For a nonempty closed set $\mathcal{C} \subset \mathbb{R}^{d}$, the distance of a point $x \in \mathbb{R}^{d}$ from $\mathcal{C}$ is defined by $\text{dist}(x, \mathcal{C}) = \inf_{y \in \mathcal{C}} \|x - y\|$. We also denote $\Delta_0 \defeq F\left(x^0\right) - F^*$ and $F^* \defeq \min _{x \in \mathbb{R}^d} F(x)$. The characteristic function of a set $A$ is denoted as $\chi_A$. 

\section{PRELIMINARIES}
First, we formally state useful definitions of the strong convexity and smoothness of a function.
\begin{definition}
    Function $f: \mathbb{R}^n \rightarrow \mathbb{R}$ is $L$-smooth if:
\begin{align*}
\|\nabla f(x)-\nabla f(y)\| \leq L\|x-y\| \: \forall \: x, y \in \mathbb{R}^n.
\end{align*}
\end{definition}
\begin{definition}
    Function $f: \mathbb{R}^n \rightarrow \mathbb{R}$ is  $\mu$-strongly convex if:
\begin{align*}
f(y) \geq f(x) + \nabla f(x)^{\top}(y-x) +\mu\|y - x\|^2/2 \:\forall \:x, y \in \mathbb{R}^n.\end{align*}
\end{definition}
Without loss of generality, we can assume that $G+r_1 $ and $H + r_2$ are strongly convex since $F$ can be reformulated as a difference of two strongly convex functions for any $\rho>0$
\begin{align*}
    F=\left(G+r_1 + \rho\|\cdot\|^2/2 \right)-\left(H+r_2+ \rho\|\cdot\|^2/2 \right).
\end{align*}
\begin{assumption}
    \label{ass:strong_cvx}
    The functions $G+r_1$ and $H + r_2$ are $\rho_{G+r_1}$-strongly-convex and $\rho_{H + r_2}$-strongly-convex, respectively. 
\end{assumption}
Next, for efficient stochastic DCA \cite{le2022stochastic} or MM \cite{phan2024stochastic}, it is common to assume that each component of the finite-sum $h_i$ is $L$-smooth. We \textit{weaken} this assumption with the following average $L$-smoothness \cite{zhou2019lower, li2021page, nguyen2022finite}.
\begin{assumption}
    \label{ass:L_average_smooth}
    The function $H(x)$ is average $L$-smooth if there exists $L \in \mathbb{R}_{>0}$ such that:
\begin{align*}
\mathbb{E}_i\left[\left\|\nabla h_i(x)-\nabla h_i(y)\right\|^2\right] \leq L^2\|x-y\|^2 \: \forall \: x, y \in \mathbb{R}^n.
\end{align*}
\end{assumption}
From \cite[Lemma 1]{li2021page}, if $H(x)$ is average $L$-smooth, then $H$ is also $L$-smooth. Now, similar to \cite{le2022stochastic}, we need the smoothness of $r_2$ for establishing a computational complexity.
\begin{assumption}
    \label{ass:r_2_smooth}
    $r_2(x)$ is $L_{r_2}$-smooth.
\end{assumption}

Finally, we consider the standard measure of optimality for the DC problem \eqref{eq:prob_main} \cite{le2022stochastic}.
\begin{definition}[Critical Distance Measure]
    \label{def:dist} 
    We define the measure of proximity to criticality as
    \begin{align*}
        d(x^t)= \mathbb{E} \left[\operatorname{dist}\left(\nabla\left(H+r_2\right)\left(x^{t}\right), \partial\left(G+r_1\right)\left(x^{t}\right)\right)\right].
    \end{align*}
    We have that $x^*$ is a critical point of problem (\ref{eq:prob_main}) if $d(x^\ast) = 0$. 
    We also define $\hat x$ to be a DC $\varepsilon$-critical solution if $d(\hat x) \leq \varepsilon$.
\end{definition}

\section{DCA-PAGE ALGORITHM}
We summarize DCA-PAGE for \eqref{eq:prob_main} in Algorithm \ref{alg:dca_page}. In brief, in the $t$-th iteration, we first compute the PAGE-like gradient estimator $g^t$ for $H(x^t)$. While we require the smoothness of $r_2$ (Assumption \ref{ass:r_2_smooth}) for the analysis in Section \ref{sec:conv_analysis}, we extend the analysis to \textit{non-differentiable} $r_2$ in Section \ref{sec:gap_analysis}. Therefore, we only compute a subgradient $w^t  \in \partial r_2 (x^t)$ in Algorithm \ref{alg:dca_page}. We then use the gradient estimator $g^t$ and $w^t$ to construct the convex subproblem and solve for $x^{t+1}$. We then choose an output $\hat x$ uniformly at random from $\{x^i\}^T_{i=1}$.
\begin{algorithm}[ht]
\caption{DCA-PAGE}
\label{alg:dca_page}
\KwIn{$x^0 \in \text{dom} \, r_1$, $b'<b$, probabilities $\{p_t\}_{t\in T}$.}
Compute $g^0 = \frac{1}{b} \sum_{i \in I} \nabla h_i (x^0)$.
\For{\(t = 1,2,\dots,T\)}{
    Compute PAGE-like gradient \( g^{t} \) of \( H(x^t) \):
    \begin{align*}
        g^{t} = \begin{cases} &\frac{1}{b} \sum_{i \in I} \nabla h_i (x^t) \\
            & \text{w.p. } p_t, \\
            &\text{\small{\text{$g^{t-1} + \frac{1}{b'} \sum_{i \in I'} \left( \nabla h_i (x^t) - \nabla h_i (x^{t-1}) \right)$}}} \\
            & \text{w.p. } 1 - p_t.
        \end{cases}
    \end{align*}\\
    Compute $w^t \in \partial r_2 (x^t)$\\
    Solve the convex program: 
    \begin{align*}
        \text{\small{\text{$x^{t+1} \in \argmin_{x \in \mathbb{R}^n} G(x) + r_1(x) - \langle g^t + w^t, x \rangle$}}}.
    \end{align*} 
}
\KwOut{$\hat{x}$, chosen uniformly from $\{x^i\}^T_{i=1}$.}
\end{algorithm}

PAGE achieves the optimal gradient computation complexity for smooth nonconvex optimization, whereas SVRG and SAGA exhibit suboptimal complexity \cite{li2021page}, so it motivates us to replace SAGA and SVRG with PAGE for SVR DCA methods. Moreover, compared to the loopless SARAH variant in MM-SARAH \cite{phan2024stochastic}, which involves periodic full gradient computations, PAGE permits using sufficiently large batches and thus offers greater flexibility. DCA-PAGE hence can be readily extended for online settings, where computing full gradients for an arbitrarily large \(N\) is infeasible.
\section{CONVERGENCE ANALYSIS}
\label{sec:conv_analysis}
We will first prove the following useful template inequality for an adjustable $\gamma \in (0,\rho)$.  
\begin{lemma}
    \label{lem:template}
    Given $\gamma\in(0,\rho)$. For $\rho = \rho_{H+r_2}+\rho_{G+r_1}$, $x^t$ and $x^{t+1}$ are iterates of Algorithm \ref{alg:dca_page} for \eqref{eq:prob_main}, we have 
    \begin{align}
    F\left(x^{t+1}\right) &\leq F\left(x^t\right)+ \|g^{t}-\nabla H\left(x^t\right)\|^2 / (2\gamma)  \notag \\
    &-(\rho-\gamma)\left\|x^{t+1} - x^{t} \right\|^2 / 2.
    \label{eq:F_main_bound}
\end{align}
\end{lemma}

We now present the guarantees for DCA-PAGE in both the finite-sum and online settings.

\subsection{Finite-sum Setting}
\begin{theorem}
\label{them:covergence_dca_page_finite_sum}
Under Assumptions \ref{ass:strong_cvx}, \ref{ass:L_average_smooth}, \ref{ass:r_2_smooth}, we choose minibatch sizes $b=N$, $b^{\prime}<N^{1/2}$, probability $p_t = p = N^{-1/2}\in(0,1]$, and $\rho = \rho_{G+r_1} + \rho_{H + r_2}$ . Then, the number of iterations performed by DCA-PAGE sufficient for finding a DC $\varepsilon$-critical solution is
\begin{align*}
    T = \frac{8\Delta_0}{\varepsilon^2 \rho}  \left( \left(L+L_{r_2}\right)^2 + \frac{(1-p)L^2}{pb'} \right),
\end{align*}
with the following count of stochastic gradient computations: 
\begin{align*}
    b + T(pb + (1-p)b') = \mathcal{O}\left(N + N^{1/2} \varepsilon^{-2} \right).
\end{align*}
\end{theorem}

\textbf{Conditions for $b'$}. Similar to \cite{le2022stochastic, phan2024stochastic}, we require the following condition for the validity of $b'$
\begin{align*}
    \left\lceil N^{1/2} \right\rceil \geq b' \geq \frac{4(1-p)L^2}{p \rho^2} \Rightarrow  \frac{L}{ \rho} \leq \sqrt{\frac{p\left\lceil N^{1/2}\right\rceil}{4(1-p)}}.
\end{align*}
In large-scale applications, i.e. $N$ is very large, this condition is naturally satisfied. For the small-scale DC problem that does not satisfy the above relation, we can actually adjust the DC decomposition to decrease $L/ \rho$ accordingly. Indeed, we can add $\frac{\eta}{2}\|\cdot\|^2$, where $\eta>0$, to both DC components:
\begin{align*}
    F_{\text{new}}=\left(G+r_1+ \eta  \|\cdot\|^2/ 2 \right)-\left(H+r_2+ \eta  \|\cdot\|^2/2 \right).
\end{align*}
With this new DC decomposition, we can decrease the ratio $L_\text{new} / \rho_\text{new} = L / (\rho+2 \eta),$ by increasing $\eta$. 

\subsection{Online Setting}
Recall that in the online/streaming regime, we model
\[
H(x)\defeq \mathbb{E}_{\zeta\sim\mathcal{D}}[\Phi(x,\zeta)],
\]
where $\zeta$ is a random data instance and $\Phi$ is the instantaneous sample loss. In the analysis, we write $h_i(x)\defeq \Phi(x,\zeta_i)$ for i.i.d.\ samples $\{\zeta_i\}$ arriving from the stream. We now require the following extra assumption
\begin{assumption}
    \label{ass:bounded_var}
    The stochastic gradient has bounded variance if there exists $ \sigma>0$, such that for all $x \in \mathbb{R}^d$
\begin{align*}
\mathbb{E}_i\left[\left\|\nabla h_i(x)-\nabla H(x)\right\|^2\right] \leq \sigma^2. 
\end{align*}
\end{assumption}
Bounded variance ensures that the noise of stochastic gradient estimators remains controlled, which is essential for maintaining stability when full gradient is unavailable. We now present the guarantee in the online setting. 
\begin{theorem}
\label{them:covergence_dca_page_online}
Under Assumptions \ref{ass:strong_cvx}, \ref{ass:L_average_smooth}, \ref{ass:r_2_smooth}, \ref{ass:bounded_var}, we choose minibatch size $b = \lceil\sigma^2/(\alpha \varepsilon^2 )\rceil$, where $\alpha =\frac{\rho L}{4(4C  +\rho L)}$, $C = \left(L+L_{r_2}\right)^2 + L^2$, $\rho = \rho_{G+r_1} + \rho_{H + r_2}$. We then choose second minibatch size $b^{\prime}<b^{1/2}$, and probability $p_t = p = b^{-1/2} \in(0,1]$. Then, the number of iterations performed by DCA-PAGE sufficient for finding a DC $\varepsilon$-critical solution is
\begin{align*}
    T \leq \frac{16C\Delta_0}{\varepsilon^2 \rho} + \frac{2C}{p (2C + \rho\gamma)}, 
\end{align*}
with the following count of stochastic gradient computations: 
\begin{align*}
    b + T(pb + (1-p)b') = \mathcal{O} \left(b + b^{1/2} \varepsilon^{-2} \right).
\end{align*}
\end{theorem}
\begin{remark}
    Previous online DCA works \cite{le2020online, le2022online} only establish asymptotic convergence, albeit in a more general setting. Moreover, previous SVR DCA methods DCA-SAGA and DCA-SVRG \cite{le2022stochastic} that rely on periodic full gradient updates cannot be extended to the online setting when full gradient computation is infeasible. In contrast, DCA-PAGE only requires periodic updates using a sufficiently large minibatch, so it is readily extendable to online settings.
\end{remark}

\section{GAP FUNCTION AND REFINED ANALYSIS}
\label{sec:gap_analysis}
In this section, we will remove the Assumption \ref{ass:r_2_smooth} and show the guarantees for DCA-PAGE with an alternative measure of criticality. Inspired by \cite[Definition 1]{maskan2025revisiting}, we define the following notion of gap to measure the closeness to the first-order stationary point of \eqref{eq:prob_main}
\begin{definition}[Gap Function]
    Given $w^{t-1} \in \partial r_2(x^{t-1})$, we measure closeness to first-order stationary point with
\begin{align*}
    \operatorname{gap} \left(x^t\right) &\defeq \max _{x \in \mathbb{R}^n} (G+r_1)\left(x^t\right)- (G+r_1)(x)\\
    &-\left\langle\nabla H\left(x^{t-1}\right)  + w^{t-1}, x^t-x\right\rangle.
\end{align*}
\end{definition}
We can verify that the $\operatorname{gap}$ function is nonnegative and equals zero if and only if $x^t$ is a first-order stationary point. Despite its recent introduction, the gap function has proven to be a more general measure than the critical distance (Definition \ref{def:dist}). The following subsection presents an example that illustrates why the gap function can be more suitable than the critical distance measure for nonsmooth problems.

\subsection{Critical Distance Measure versus Gap Function}
Modifying from the constrained example in \cite[Appendix C]{maskan2024block}, we can give a simple unconstrained nonsmooth example where the critical distance measure remain lower-bounded even when approaching the optimum arbitrarily close. Consider $\min_{x\in \mathbb{R}} |x|$, where $G(x) = |x|$, and $H(x) = r_1(x) = r_2(x) = 0$. Then for any $x'$, we have
\begin{align*}
    d(x') &= \operatorname{dist}\left(\nabla\left(H+r_2\right)\left(x'\right), \partial\left(G+r_1\right)\left(x'\right)\right)\\
    &= \min_{w \in \partial |x|} |0 - w| \\
    &= \begin{cases}
        0 & \text{when $x' = 0$}\\
        1 & \text{otherwise}
    \end{cases}.
\end{align*}
Hence, it is equal $0$ only at the global optimum and remains $1$ at \textit{every other point}, including those arbitrarily close to the optimum. On the other hand, the gap function is
\begin{align*}
    \operatorname{gap} (x') = \max_{x \in \mathbb{R}} |x'| - |x| = |x'|,
\end{align*}
which goes to $0$ as $x'$ goes to $0$. This example shows how gap function can be a much better measure than the distance to criticality measure for nonsmooth problems. 
\subsection{Analysis with Gap Function}
We first change the proof for the template inequality Lemma \ref{lem:template} under this milder setting. 

Recall that $\hat{x}$ is the output of Algorithm \ref{alg:dca_page}, we discuss the following key lemma
\begin{lemma}
    \label{lem:key}
    Suppose Assumptions \ref{ass:strong_cvx}, \ref{ass:L_average_smooth} hold. For iterates $x^t$ of DCA-PAGE (Algorithm \ref{alg:dca_page}), we have
    \begin{align*}
    \operatorname{gap}(\hat{x}) \leq  \frac{\left(1-p\right) L^2}{2 \eta p b^{\prime} T}\sum_{t=0}^{T-1} \mathbb{E}\left[\left\|x^{t+1}-x^t\right\|^2 \right]. 
\end{align*}
\end{lemma}
We now present the guarantee in the finite-sum setting
\begin{theorem}
\label{them:dca_page_gap_fin_sum}
Under Assumptions \ref{ass:strong_cvx}, \ref{ass:L_average_smooth}, we choose minibatch sizes $b=N$, $b^{\prime}<N^{1/2}$, probability $p_t = p = N^{-1/2}\in(0,1]$, and $\rho = \rho_{G+r_1} + \rho_H + \rho_{r_2}$ . Then, the number of iterations performed by DCA-PAGE to obtain $\operatorname{gap}(\hat{x}) \leq \varepsilon$ is 
\begin{align}
        T = \frac{2 \left(1-p\right) L^2 \Delta_0}{ \rho_{G+r_1} p b^{\prime} \rho \varepsilon}
        \label{eq:refine_T_finite}
    \end{align}
with the following count of stochastic gradient computations:
\begin{align*}
    b + T(pb + (1-p)b') = \mathcal{O}\left(N + N^{1/2} \varepsilon^{-1} \right).
\end{align*}
\end{theorem}

Next, we also present the convergence for DCA-PAGE in the online setting with the gap function
\begin{theorem}
\label{them:dca_page_gap_onl}
Under Assumptions \ref{ass:strong_cvx}, \ref{ass:L_average_smooth}, \ref{ass:bounded_var}, we choose minibatch sizes $b = \lceil\sigma^2/(\alpha \varepsilon)\rceil$, where $ \alpha = \frac{\rho_{G+r_1} \rho \gamma p}{2L^2}$, $b^{\prime}<b^{1/2}$, probability $p_t = p = b^{-1/2}\in(0,1]$, and $\rho = \rho_{G+r_1} + \rho_H + \rho_{r_2}$. Then, the number of iterations performed by DCA-PAGE to obtain $\operatorname{gap}(\hat{x}) \leq \varepsilon$ is 
\begin{align}
        T = \frac{4 \left(b'-1\right) L^2}{ \rho_{G+r_1} b^{\prime} \rho  \varepsilon} \left(\Delta_0 + \frac{\alpha \varepsilon}{2p\gamma}\right)
        \label{eq:refine_T_online}
\end{align}
with the following count of stochastic gradient computations:
\begin{align*}
    b + T(pb + (1-p)b') = \mathcal{O}\left(b + b^{1/2} \varepsilon^{-1}  \right).
\end{align*}
\end{theorem}
Although the gap function for DCA-type algorithms has been explored in other contexts \cite{maskan2024block, maskan2025revisiting}, this work provides the first analysis of the gap function in a stochastic and unconstrained setting. 
\begin{remark}
    Note that while the complexities in this section are in the order $\mathcal{O}(\varepsilon^{-1})$, they are actually the same with the complexities with $\mathcal{O}(\varepsilon^{-2})$ in the previous section. This equivalence arises because the gap function is on the same order as the square of the critical distance measure.
\end{remark}

\section{NUMERICAL EXPERIMENTS}
For the numerical experiments, we consider nonconvex binary and multi-class logistic classification similar to \cite{phan2024stochastic}. 

Let \(\{(a_i, b_i)\}\), for $i\in [n]$ be a training set with observation vectors 
\(a_i \in \mathbb{R}^d\) and labels \(b_i \in \{-1, 1\}\). We consider the nonconvex loss function
\begin{align*}
\min_{w \in \mathbb{R}^d} \frac{1}{n} \sum_{i=1}^n \phi\left(a_i^\top w, b_i\right) + r(w),
\end{align*}
where $\phi(a_i^\top w, b_i) = \left(1 - (1 + \exp(-a_i^\top w b_i))^{-1} \right)^2$~\cite{zhao2010convex}, $r(w) = \sum_{j=1}^d \lambda (1 - e^{-\alpha |w_j|})$ is the exponential regularizer with  parameters $\lambda, \alpha > 0$~\cite{bradley1998feature}. 
From \cite{phan2024stochastic}, $\phi(a_i^\top w, b_i)$ is $L$-smooth with $L = (39+55\sqrt{33}) \max_{i\in[n]} \|a_i\|^2 / 2304$. As the regularizer $r$ is concave, it has the trivial DC decomposition. Hence, we can formulate a DC program in a form similar to \eqref{eq:prob_spec}. We then apply DCA-PAGE, DCA-SVRG \cite{le2022stochastic}, DCA-SAGA \cite{le2022stochastic}, Stochastic DCA (SDCA) \cite{le2020stochastic} to solve the DC program over 10 independent runs to obtain Figure \ref{fig:bin} with two datasets a9a and w8a from LIBSVM \cite{chang2011libsvm}. It is evident that DCA-PAGE improves the practical performance of previous SVR DCA methods as the theory suggests. 
\begin{figure}
    \centering
    \includegraphics[width=\linewidth]{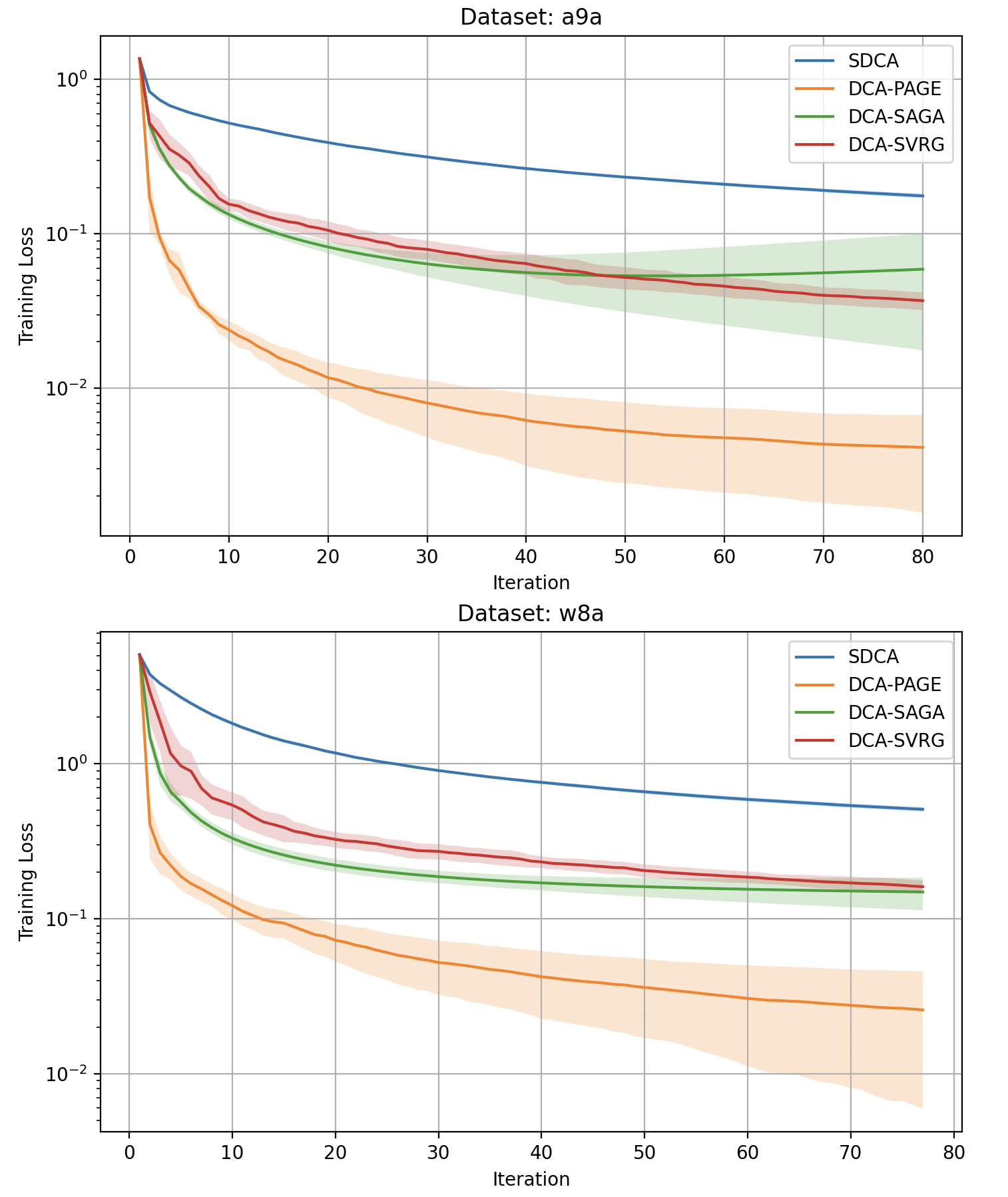}
    \caption{SVR DCA methods for binary classification}
    \label{fig:bin}
\end{figure}

We also consider the setting of multi-class logistic regression. Let  \(c \) be the number of classes, and  \( b_i \in \{1, 2, \dots, c\} \) be the labels. Consider the nonconvex loss function:
\begin{align*}
\min_{W \in \mathbb{R}^{d \times c}} \frac{1}{n} \sum_{i=1}^n \phi\left(a_i, W, b_i\right) + r(W),
\end{align*}
where $\phi\left(a_i, W, b_i\right) = \log \left( \sum_{k=1}^c \exp(a_i^\top W_k) \right) - a_i^\top W_{b_i}$ as \( W_k \) is the \( k \)-th column of \( W \),
and $r(W) = \sum_{j=1}^d \lambda (1 - e^{-\alpha \| W_j^\top \|})$ with parameters $\lambda, \alpha > 0$. As $\phi\left(a_i, W, b_i\right)$ is $L$-smooth with $L = \frac{c-1}{c} \max_{i \in [n]} \|a_i\|^2$, we can also rewrite this multi-class logistic regression as a DC program in the form (\ref{eq:prob_spec}). We also apply DCA-PAGE, DCA-SVRG, DCA-SAGA, Stochastic DCA (SDCA) to solve this DC program over 10 independent runs to obtain Figure \ref{fig:multi} with the dataset dna from LIBSVM \cite{chang2011libsvm}. DCA-PAGE again shows better training loss than the other stochastic DCA methods. 
\begin{figure}
    \centering
    \includegraphics[width=\linewidth]{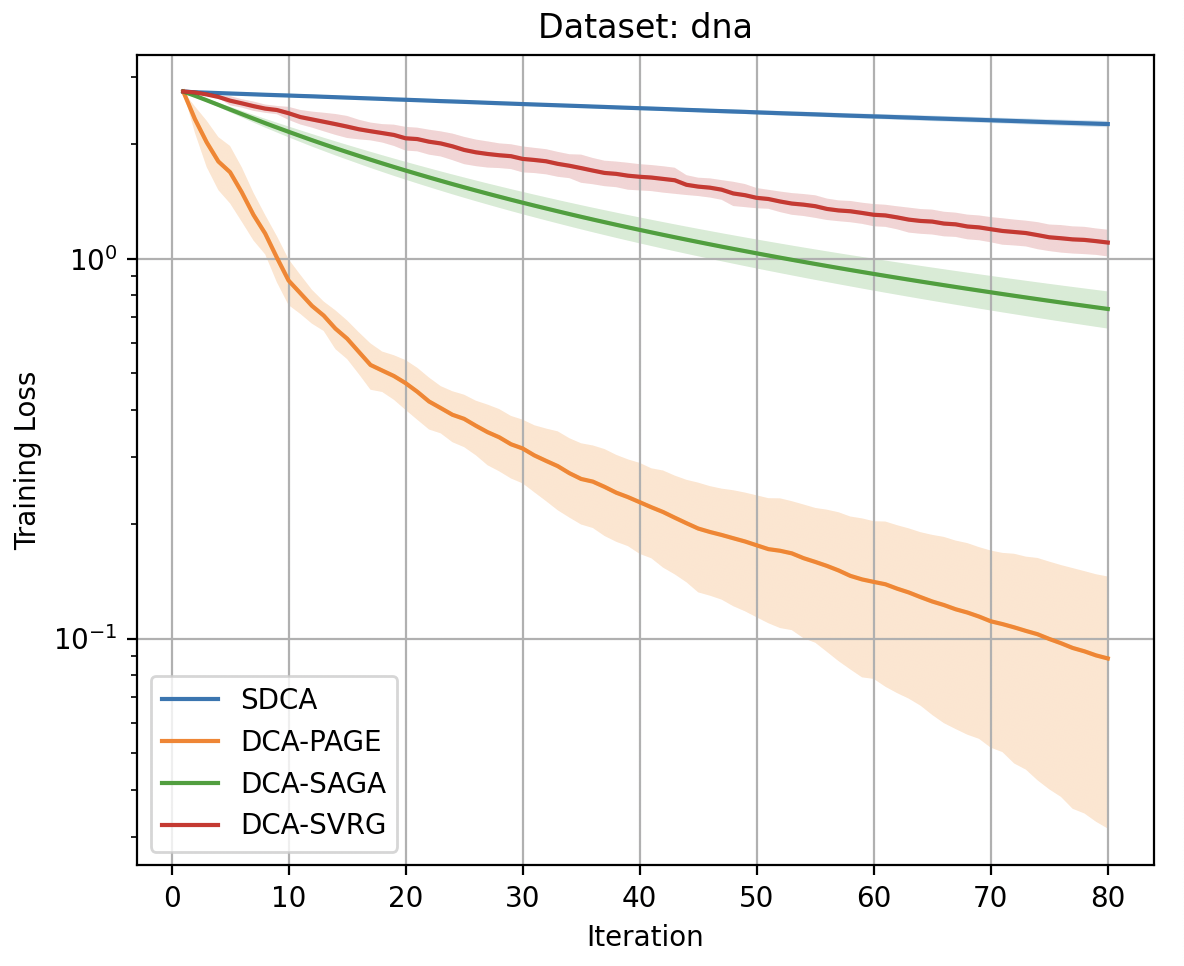}
    \caption{SVR DCA methods for multi-class classification}
    \label{fig:multi}
\end{figure} 
\section{CONCLUSIONS AND FUTURE WORK}
We propose DCA-PAGE which achieves a gradient computation complexity of \(\mathcal{O}(N + N^{1/2}\varepsilon^{-2})\) in the finite-sum setting, outperforming previous SVR DCA methods. Moreover, our method readily extends to online settings with the same optimal complexity, thereby overcoming a key limitation of earlier approaches. Additionally, by refining our analysis with the gap function, we obtained similar convergence guarantees under milder assumptions.

Future research can further enhance the practical impact of DCA-PAGE by investigating large-scale applications such as nonnegative matrix factorization and phase retrieval. Another potential future direction is to explore the convex subproblem with SVR methods on top of constructing the subproblem with stochastic gradient estimators, as explored in this paper.

\section{ACKNOWLEDGMENTS}
AD Nguyen and KC Toh are supported by the Ministry of Education, Singapore, under its 2019 Academic Research Fund Tier 3 grant call (Award ref: MOE-2019-T3-1-010). A Yurtsever was supported by the Wallenberg AI, Autonomous Systems, and Software Program (WASP), funded by the Knut and Alice Wallenberg Foundation. S Sra acknowledges generous support from the Alexander von Humboldt Foundation. 

\addtolength{\textheight}{-3cm}   

\bibliographystyle{IEEEtran}
\bibliography{refs}
\onecolumn
\begin{center}
    {\LARGE Appendix}   
\end{center}
\section{Foundations for DC Programming and DCA}
We consider the \emph{DC program} \cite{le2018dc, le2024open} of the following form
\begin{equation}
\min \{f(x) := g(x) - h(x) : x \in \mathbb{R}^n\} 
\tag{$P_{DC}$}
\label{eq:p_dc}
\end{equation}
where the functions $g$ and $h$ are convex. This function $f$ is called a DC function on $\mathbb{R}^n$, and $g-h$ is called a DC decomposition of $f$, while $g$ and $h$ are its DC components. With the convention $+\infty - (+\infty) = +\infty$, the finiteness of the optimal value of \eqref{eq:p_dc} implies $\mathrm{dom}\, f = \mathrm{dom}\, g \subseteq \mathrm{dom}\, h$ \cite{le2018dc, le2024open}. 

Note that, for a nonempty closed convex set $C \subset \mathbb{R}^n$, the constrained DC program
\begin{equation*}
\min \{f(x) := g(x) - h(x) : x \in C\}
\end{equation*}
can be reformulated straightforwardly as \eqref{eq:p_dc} 
\[
\min \{(g + \chi_C)(x) - h(x) : x \in \mathbb{R}^n\}.
\]
The common notion of criticality of DC programming \cite{le2018dc, le2024open} is as follows 
\begin{definition} 
    We call $x^* \in \mathrm{dom}\, g$ a DC-critical point of $g - h$ if
    \[
\partial g(x^*) \cap \partial h(x^*) \neq \emptyset \text{ or equivalently
 } 0 \in [\partial g(x^*) - \partial h(x^*)],
\]
i.e., a zero of the difference of two cyclically maximal monotone $\partial g$ and $\partial h$.
\end{definition}

The standard DCA (Algorithm \ref{alg:dca}) was introduced to solve the standard DC program \eqref{eq:p_dc} \cite{tao1997convex}. At the $t$-th iterate, DCA simply consists in linearizing the second DC component $h$ to get the affine minorant $\bar{h}^t(x) = h(x^t) + \inner{y^t}{x-x^t}$, where $y^t$ is a subgradient of $h(x^t)$. Then the next iterate $x^{t+1}$ is a solution to the convex subprogram with the affine minorant
\begin{align*}
    x^{t+1} &\in \argmin_{x \in \mathbb{R}^n} g(x)-  h(x^t) - \inner{y^t}{x-x^t}\\
    &= \argmin_{x \in \mathbb{R}^n} g(x)- \langle y^t, x \rangle.
\end{align*}
The algorithm reiterates until the stopping or convergence criteria is met.
\begin{algorithm}[ht]
\caption{Standard DCA}
\label{alg:dca}
\KwIn{$x^0 \in \text{dom} \, r_1$.}
\For{\(t = 1,2,\dots,T\)}{
    Compute $y^t \in \partial h (x^t)$.\\
    Solve for $x^{t+1}$: 
    \begin{align*}
       x^{t+1} \in \argmin_{x \in \mathbb{R}^n} g(x)- \langle y^t, x \rangle.
    \end{align*} 
}
\KwOut{$\hat{x}$, chosen according to the stopping condition.}
\end{algorithm}

An important subclass of DCA for smooth optimization is the Concave-Convex Procedure (CCCP) \cite{yuille2003concave,lipp2016variations}. When $f$ and $g$ are differentiable, we can change the subgradient $y^t$ to $\nabla h (x^t)$. The convex subprogram then becomes
\begin{align*}
    x^{t+1} \in \argmin_{x \in \mathbb{R}^n} g(x)- \langle \nabla h (x^t), x \rangle,
\end{align*}
which implies by the first order optimality condition that $\nabla g(x^{t+1}) = \nabla h (x^t)$.

Regarding the general convergence property, DCA has a sequential convergence to a DC critical point in the sense that every convergent subsequence of $\{x^t\}$ has a critical point of \eqref{eq:p_dc} as its limit \cite[Appendix A]{le2024open}. The well-known complexity of DCA is $\mathcal{O} (1/T)$ \cite[Theorem 11]{le2024open}.
\section{Proof for Lemma \ref{lem:template}}
Since $H+r_2$ are $\rho_{H+r_2}$-convex, we have that
\begin{align*}
(H+r_2)\left(x^{t+1}\right) &\geq (H+r_2)\left(x^t\right) + \frac{\rho_{H+r_2}}{2}\left\|x^{t+1}-x^{t}\right\|^2 + \left\langle\nabla H\left(x^t\right) + \nabla r_2 (x^t), x^{t+1}-x^{t}\right\rangle. 
\end{align*}
From the optimality of $x^{t+1}$, we have that $g^t + \nabla r_2 (x^t) \in \partial (G+r_1)$. By the strong convexity of $(G+r_1)$, we have 
\begin{align*}
\left(G+r_1\right)\left(x^t\right) &\geq\left(G+r_1\right)\left(x^{t+1}\right)+\frac{\rho_{G+r_1}}{2}\left\|x^{t} - x^{t+1}\right\|^2 +\left\langle g^t + w^t, x^t-x^{t+1}\right\rangle.
\end{align*}
Adding the above inequalities, we obtain
\begin{align*}
F\left(x^{t+1}\right) &\leq F\left(x^t\right) - \frac{\rho}{2} \left\|x^{t+1} - x^{t} \right\|^2 +\left\langle x^{t+1}-x^t, g^{t}-\nabla H\left(x^t\right)\right\rangle,
\end{align*}
where $\rho=\rho_{H+r_2}+\rho_{G+r_1}$. Since $\inner{a}{b} \leq \frac{1}{2\gamma} \|a\|^2 + \frac{\gamma}{2}\|b\|^2$, we obtain the desired inequality. \qed
\section{Proof for Theorem \ref{them:covergence_dca_page_finite_sum}}
First, we note the variance bound for PAGE gradient estimator from \cite{li2021page}:
\begin{lemma} \cite[Lemma 3]{li2021page}
\label{lem:page_bound}
    Suppose that $H$ is average $L$-smooth. For PAGE gradient estimator $g^{t+1}$, we have that
\begin{align}
\mathbb{E}\left[\left\|g^{t+1}-\nabla H\left(x^{t+1}\right)\right\|^2\right] \leq\left(1-p_t\right)\left\|g^t-\nabla H\left(x^t\right)\right\|^2+\frac{\left(1-p_t\right) L^2}{b^{\prime}}\left\|x^{t+1}-x^t\right\|^2.
\label{eq:page_bound}
\end{align}
\end{lemma}
In this setting, we set $p_t = p$. We then have the following corollary
\begin{corollary}
    \label{cor:page_sum}
    In a similar setting to Lemma \ref{lem:page_bound}, we also have 
    \begin{align*}
        \sum_{t=0}^{T-1} \mathbb{E}\left[\left\|g^t-\nabla H\left(x^t\right)\right\|^2\right] \leq \frac{\left(1-p\right) L^2}{p b^{\prime}}\sum_{t=0}^{T-1} \mathbb{E}\left[\left\|x^{t+1}-x^t\right\|^2 \right].
    \end{align*}
\end{corollary}
\begin{proof}
    Summing both sides of \eqref{eq:page_bound} for $t = 0, \ldots, T-1$, we can take expectation to obtain
    \begin{align*}
        \sum_{t=0}^{T-1} \mathbb{E}\left[\left\|g^{t+1}-\nabla H\left(x^{t+1}\right)\right\|^2\right] \leq \left(1-p\right)\sum_{t=0}^{T-1} \mathbb{E}\left[\left\|g^t-\nabla H\left(x^t\right)\right\|^2\right]+\frac{\left(1-p\right) L^2}{b^{\prime}}\sum_{t=0}^{T-1} \mathbb{E}\left[\left\|x^{t+1}-x^t\right\|^2 \right],
    \end{align*}
    which implies 
    \begin{align*}
        p\sum_{t=0}^{T-1} \mathbb{E}\left[\left\|g^t-\nabla H\left(x^t\right)\right\|^2\right] + \mathbb{E}\left[\left\|g^{T}-\nabla H\left(x^{T}\right)\right\|^2 \right] \leq \frac{\left(1-p\right) L^2}{b^{\prime}}\sum_{t=0}^{T-1} \mathbb{E}\left[\left\|x^{t+1}-x^t\right\|^2 \right].
    \end{align*}
Taking away the positive term $\mathbb{E}\left[\left\|g^{T}-\nabla H\left(x^{T}\right)\right\|^2 \right]$, we then have 
    \begin{align*}
        \sum_{t=0}^{T-1} \mathbb{E}\left[\left\|g^t-\nabla H\left(x^t\right)\right\|^2\right] \leq \frac{\left(1-p\right) L^2}{p b^{\prime}}\sum_{t=0}^{T-1} \mathbb{E}\left[\left\|x^{t+1}-x^t\right\|^2 \right].
    \end{align*}
\end{proof}

Next, we add \eqref{eq:F_main_bound} to $\frac{1}{2\gamma p}$ \eqref{eq:page_bound} and take expectation to obtain
\begin{align*}
& \mathbb{E} {\left[F\left(x^{t+1}\right)-F^*+\frac{1}{2 p \gamma}\left\|g^{t+1}-\nabla H\left(x^{t+1}\right)\right\|^2\right] } \\
& \leq \mathbb{E}\left[F\left(x^t\right)-F^*-\frac{\rho-\gamma}{2}\left\|x^{t+1}-x^t\right\|^2+\frac{1}{2 \gamma}\left\|g^t-\nabla H\left(x^t\right)\right\|^2\right] \\
&+\frac{1}{2 p\gamma} \mathbb{E}\left[(1-p)\left\|g^t-\nabla H\left(x^t\right)\right\|^2+\frac{(1-p) L^2}{b^{\prime}}\left\|x^{t+1}-x^t\right\|^2\right] \\
&=\mathbb{E}\left[F\left(x^t\right)-F^*+\frac{1}{2 p \gamma}\left\|g^t - \nabla H\left(x^t\right)\right\|^2 - \left(\frac{\rho-\gamma}{2}-\frac{(1-p) L^2}{2 p b^{\prime}\gamma}\right)\left\|x^{t+1}-x^t\right\|^2\right] 
\end{align*}
Let $\Omega^t \defeq F\left(x^t\right)-F^*+\frac{1}{2 p \gamma}\left\|g^t - \nabla H\left(x^t\right)\right\|^2$. We now have
\begin{align*}
    \mathbb{E} [\Omega^{t+1}] \leq \mathbb{E} [\Omega^{t}] - \left(\frac{\rho-\gamma}{2}-\frac{(1-p) L^2}{2 p b^{\prime}\gamma}\right)\mathbb{E}\left[ \left\|x^{t+1}-x^t\right\|^2 \right],
\end{align*}
which implies
\begin{align*}
    \left(\frac{\rho-\gamma}{2}-\frac{(1-p) L^2}{2 p b^{\prime}\gamma}\right)\mathbb{E}\left[ \left\|x^{t+1}-x^t\right\|^2 \right] \leq \mathbb{E} [\Omega^{t}]  - \mathbb{E} [\Omega^{t+1}].
\end{align*}
We can optimize and choose the second batch size $b'$ 
\begin{align}
    \label{eq:b'_lower_bound}
    b' \geq \frac{4(1-p)L^2}{p \rho^2},
\end{align}
and the tunable parameter $\gamma$
\begin{align}
\label{eq:gamma_choice}
\gamma = L \sqrt{\frac{1-p}{pb'}},
\end{align}
to ensure that 
\begin{align*}
    \frac{\rho-\gamma}{2}-\frac{(1-p) L^2}{2 p b^{\prime}\gamma} = \frac{\rho}{2} - \gamma = \frac{\rho}{2}  - L \sqrt{\frac{1-p}{pb'}} \geq \frac{\rho}{2}  - L \sqrt{\frac{1-p}{4(1-p)L^2 / \rho^2}} = 0
\end{align*}
We note that any choice of $b'$ that satisfies \eqref{eq:b'_lower_bound} will suffice, and the final result will only be different by a constant. For notation convenience, we just pick 
\begin{align}
    \label{eq:b'_choice}
    b' = \frac{16(1-p)L^2}{p \rho^2},
\end{align}
which implies that 
\begin{align*}
    \frac{\rho-\gamma}{2}-\frac{(1-p) L^2}{2 p b^{\prime}\gamma} = \frac{\rho}{4}.
\end{align*}
By telescoping sum, we obtain
\begin{align*}
   \frac{\rho}{4} \sum_{t=0}^{T-1}\mathbb{E}\left[ \left\|x^{t+1}-x^t\right\|^2 \right] \leq \mathbb{E} [\Omega^{0}]  - \mathbb{E} [\Omega^{T}] \leq \mathbb{E} [\Omega^{0}] = F\left(x^0\right)-F^* = \Delta_0,
\end{align*}
which leads to
\begin{align}
    \label{eq:iter_diff_sum}
    \sum_{t=0}^{T-1}\mathbb{E}\left[ \left\|x^{t+1}-x^t\right\|^2 \right] \leq \frac{4\Delta_0}{\rho}.
\end{align}
Since $g^{t} + \nabla r_2(x^t) \in$ $\partial\left(G+r_1\right)\left(x^{t+1}\right)$, we have by the smoothness of $H$ and $r_2$
\begin{align*}
\mathbb{E} \left[\operatorname{dist}\left(\nabla H\left(x^{t+1}\right)+\nabla r_2\left(x^{t+1}\right), \partial\left(G+r_1\right)\left(x^{t+1}\right)\right) \right]& \leq \mathbb{E}\left[\left\|\nabla H\left(x^{t+1}\right)-g^{t}\right\|\right]+\mathbb{E}\left[\left\|\nabla r_2\left(x^{t+1}\right)-\nabla r_2\left(x^{t}\right)\right\|\right] \\
& \leq \mathbb{E}\left[\left\|\nabla H\left(x^{t+1}\right)-\nabla H\left(x^{t}\right)\right\|\right] + \mathbb{E}\left[\left\|\nabla H\left(x^{t}\right)-g^{t}\right\|\right]\\
&+ \mathbb{E}\left[\left\|\nabla r_2\left(x^{t+1}\right)-\nabla r_2\left(x^t\right)\right\| \right] \\
& \leq\left(L+L_{r_2}\right) \mathbb{E}\left[\left\|x^{t+1}-x^t\right\| \right]+ \mathbb{E}\left[\left\|\nabla H\left(x^{t}\right)-g^{t}\right\|\right].
\end{align*}
This further implies that
\begin{align*}
    \mathbb{E} \left[\operatorname{dist}^2\left(\nabla H\left(x^{t+1}\right)+\nabla r_2\left(x^{t+1}\right), \partial\left(G+r_1\right)\left(x^{t+1}\right)\right) \right] \leq 2\left(L+L_{r_2}\right)^2 \mathbb{E}\left[\left\|x^{t+1}-x^t\right\|^2 \right]+ 2\mathbb{E}\left[\left\|\nabla H\left(x^{t}\right)-g^{t}\right\|^2\right].
\end{align*}
Summing over $t = 0,\ldots,T-1$ and following Corollary \ref{cor:page_sum} we have
\begin{align*}
     &\sum_{t=0}^{T-1} \mathbb{E} \left[\operatorname{dist}^2\left(\nabla H\left(x^{t+1}\right)+\nabla r_2\left(x^{t+1}\right), \partial\left(G+r_1\right)\left(x^{t+1}\right)\right) \right] \\
     &\leq 2 \left(L+L_{r_2}\right)^2 \sum_{t=0}^{T-1} \mathbb{E}\left[\left\|x^{t+1}-x^t\right\|^2 \right]+ 2 \sum_{t=0}^{T-1}\mathbb{E}\left[\left\|\nabla H\left(x^{t}\right)-g^{t}\right\|^2\right] \\
     &\leq \left(2\left(L+L_{r_2}\right)^2  + \frac{2(1-p)L^2}{pb'} \right)\sum_{t=0}^{T-1}\mathbb{E}\left[\left\|x^{t+1}-x^t\right\|^2 \right] \\
     &\leq \left(2 \left(L+L_{r_2}\right)^2  + \frac{2(1-p)L^2}{pb'} \right) \frac{4\Delta_0}{\rho}.
\end{align*}
Therefore we have
\begin{align*}
    d(\hat{x})^2 &\leq \frac{1}{T} \sum_{t=0}^{T-1} \mathbb{E} \left[\operatorname{dist}^2\left(\nabla H\left(x^{t+1}\right)+\nabla r_2\left(x^{t+1}\right), \partial\left(G+r_1\right)\left(x^{t+1}\right)\right) \right] \\
    &= \left(2\left(L+L_{r_2}\right)^2  + \frac{2(1-p)L^2}{pb'} \right) \frac{4\Delta_0}{\rho T} \\
    &= \left(\left(L+L_{r_2}\right)^2  + \frac{(1-p)L^2}{pb'} \right) \frac{8\Delta_0}{\rho T},
\end{align*}
which further implies
\begin{align*}
    d(\hat{x}) \leq \sqrt{ \left( \left(L+L_{r_2}\right)^2 + \frac{(1-p)L^2}{pb'} \right) \frac{8\Delta_0}{\rho T}}
\end{align*}
Therefore it takes 
\begin{align*}
    T = \frac{8\Delta_0}{\varepsilon^2 \rho}  \left( \left(L+L_{r_2}\right)^2 + \frac{(1-p)L^2}{pb'} \right) 
\end{align*}
iterations in order to obtain $d(\hat{x}) \leq \varepsilon$. 

Now by setting $p = 1 / b'$, $b = N$ and $b'= \lceil\sqrt{b} \rceil$, we obtain the number of gradient computation is
\begin{align*}
    \# \operatorname{grad} &= b + T(pb + (1-p)b') \\
    &= N + \frac{8\Delta_0}{\varepsilon^2 \rho}  \left( \left(L+L_{r_2}\right)^2 + \frac{(b' - 1)L^2}{b'} \right)  \left(\frac{b}{b'}+ \left( 1 - \frac{1}{b'} \right) b' \right) \\
    &\leq  N + \frac{8\Delta_0}{\varepsilon^2 \rho}  \left( \left(L+L_{r_2}\right)^2 + \frac{(b' - 1)L^2}{b'} \right)  2b' \\
    &=  N + \frac{16\Delta_0}{\varepsilon^2 \rho}  \left( \left(L+L_{r_2}\right)^2 b' + (b' - 1)L^2 \right) \\
    &= \mathcal{O} \left(N + \frac{\sqrt{N}}{\varepsilon^2}\right). 
\end{align*}\qed
\section{Proof of Theorem \ref{them:covergence_dca_page_online}}
We first note the variance bound for PAGE gradient estimator in the online setting from \cite{li2021page}:
\begin{lemma}
    \label{lem:page_bound_online}
     \cite[Lemma 4]{li2021page} Under the bounded variance assumption, we have
\begin{align}
\mathbb{E}\left[\left\|g^{t+1}-\nabla f\left(x^{t+1}\right)\right\|^2\right] \leq\left(1-p_t\right)\left\|g^t-\nabla f\left(x^t\right)\right\|^2+\frac{\left(1-p_t\right) L^2}{b^{\prime}}\left\|x^{t+1}-x^t\right\|^2+\mathbf{1}_{\{b<N\}} \frac{p_t \sigma^2}{b}
\label{eq:page_bound_online}
\end{align}
\end{lemma}
In this setting, we set $p_t = p$. We then have the following corollary
\begin{corollary}
    \label{cor:page_sum_online}
    In a similar setting to Lemma \ref{lem:page_bound_online}, we also have 
    \begin{align*}
        \sum_{t=0}^{T-1} \mathbb{E}\left[\left\|g^t-\nabla H\left(x^t\right)\right\|^2\right] \leq \frac{\left(1-p\right) L^2}{p b^{\prime}}\sum_{t=0}^{T-1} \mathbb{E}\left[\left\|x^{t+1}-x^t\right\|^2 \right] + \mathbf{1}_{\{b<N\}} \frac{T \sigma^2}{b}. 
    \end{align*}
\end{corollary}
\begin{proof}
    Summing both sides of \eqref{eq:page_bound_online} for $t = 0, \ldots, T-1$, we can take expectation to obtain
    \begin{small}
    \begin{align*}
        \sum_{t=0}^{T-1} \mathbb{E}\left[\left\|g^{t+1}-\nabla H\left(x^{t+1}\right)\right\|^2\right] \leq \left(1-p\right)\sum_{t=0}^{T-1} \mathbb{E}\left[\left\|g^t-\nabla H\left(x^t\right)\right\|^2\right]+\frac{\left(1-p\right) L^2}{b^{\prime}}\sum_{t=0}^{T-1} \mathbb{E}\left[\left\|x^{t+1}-x^t\right\|^2 \right] + \mathbf{1}_{\{b<N\}} \frac{p\sigma^2 T}{b},
    \end{align*}
    \end{small}which implies 
    \begin{align*}
        p\sum_{t=0}^{T-1} \mathbb{E}\left[\left\|g^t-\nabla H\left(x^t\right)\right\|^2\right] + \mathbb{E}\left[\left\|g^{T}-\nabla H\left(x^{T}\right)\right\|^2 \right] \leq \frac{\left(1-p\right) L^2}{b^{\prime}}\sum_{t=0}^{T-1} \mathbb{E}\left[\left\|x^{t+1}-x^t\right\|^2 \right] + + \mathbf{1}_{\{b<N\}} \frac{p \sigma^2 T}{b}.
    \end{align*}
Taking away the positive term $\mathbb{E}\left[\left\|g^{T}-\nabla H\left(x^{T}\right)\right\|^2 \right]$, we then have 
    \begin{align*}
        \sum_{t=0}^{T-1} \mathbb{E}\left[\left\|g^t-\nabla H\left(x^t\right)\right\|^2\right] \leq \frac{\left(1-p\right) L^2}{p b^{\prime}}\sum_{t=0}^{T-1} \mathbb{E}\left[\left\|x^{t+1}-x^t\right\|^2 \right] + + \mathbf{1}_{\{b<N\}} \frac{\sigma^2 T}{b}.
    \end{align*}
\end{proof}
Similar to the finite-sum case, we obtain the bound \eqref{eq:F_main_bound} for the objective $F$
    \begin{align}
        \label{eq:F_main_bound_online}
        F\left(x^{t+1}\right) \leq F\left(x^t\right)+ \frac{1}{2\gamma}\|g^{t}-\nabla H\left(x^t\right)\|^2  -\frac{\rho-\gamma}{2}\left\|x^{t+1} - x^{t} \right\|^2.
    \end{align}
    Next, we add \eqref{eq:F_main_bound_online} to $\frac{1}{2\gamma p}$ \eqref{eq:page_bound_online} and take expectation to obtain
\begin{align*}
& \mathbb{E} {\left[F\left(x^{t+1}\right)-F^*+\frac{1}{2 p \gamma}\left\|g^{t+1}-\nabla H\left(x^{t+1}\right)\right\|^2\right] } \\
& \leq \mathbb{E}\left[F\left(x^t\right)-F^*-\frac{\rho-\gamma}{2}\left\|x^{t+1}-x^t\right\|^2+\frac{1}{2 \gamma}\left\|g^t-\nabla H\left(x^t\right)\right\|^2\right] \\
&+\frac{1}{2 p\gamma} \mathbb{E}\left[(1-p)\left\|g^t-\nabla H\left(x^t\right)\right\|^2+\frac{(1-p) L^2}{b^{\prime}}\left\|x^{t+1}-x^t\right\|^2 + \mathbf{1}_{\{b<N\}} \frac{p \sigma^2}{b}\right] \\
&=\mathbb{E}\left[F\left(x^t\right)-F^*+\frac{1}{2 p \gamma}\left\|g^t - \nabla H\left(x^t\right)\right\|^2 - \left(\frac{\rho-\gamma}{2}-\frac{(1-p) L^2}{2 p b^{\prime}\gamma}\right)\left\|x^{t+1}-x^t\right\|^2 +  \mathbf{1}_{\{b<N\}} \frac{\sigma^2}{2b\gamma}\right] 
\end{align*}
Let $\Omega^t \defeq F\left(x^t\right)-F^*+\frac{1}{2 p \gamma}\left\|g^t - \nabla H\left(x^t\right)\right\|^2$. We now have
\begin{align*}
    \mathbb{E} [\Omega^{t+1}] \leq \mathbb{E} [\Omega^{t}] - \left(\frac{\rho-\gamma}{2}-\frac{(1-p) L^2}{2 p b^{\prime}\gamma}\right)\mathbb{E}\left[ \left\|x^{t+1}-x^t\right\|^2 \right] + \mathbf{1}_{\{b<N\}} \frac{\sigma^2}{2b\gamma}.
\end{align*}
With similar choice of $b'$ and $\gamma$ as \eqref{eq:b'_choice} and \eqref{eq:gamma_choice}, we obtain 
\begin{align*}
    \mathbb{E} [\Omega^{t+1}] \leq \mathbb{E} [\Omega^{t}] - \frac{\rho}{4}\mathbb{E}\left[ \left\|x^{t+1}-x^t\right\|^2 \right] + \mathbf{1}_{\{b<N\}} \frac{\sigma^2}{2b\gamma}.
\end{align*}
This hence implies that
\begin{align*}
    \frac{\rho}{4}\mathbb{E}\left[ \left\|x^{t+1}-x^t\right\|^2 \right] \leq  \mathbb{E} [\Omega^{t}] - \mathbb{E} [\Omega^{t+1}] + \mathbf{1}_{\{b<N\}} \frac{\sigma^2}{2b\gamma}.
\end{align*}
Summing over $t = 0, \ldots, T-1$, we obtain
\begin{align}
    \label{eq:sum_diff_iters_online}
    \frac{\rho}{4} \sum_{t=0}^{T-1} \mathbb{E}\left[ \left\|x^{t+1}-x^t\right\|^2 \right] &\leq \mathbb{E} [\Omega^{0}] - \mathbb{E} [\Omega^{T}]  + \mathbf{1}_{\{b<N\}} \frac{ \sigma^2 T}{2b\gamma} \\
    & \leq \Delta_0 + \frac{\sigma^2}{2pb\gamma}+ \mathbf{1}_{\{b<N\}} \frac{ \sigma^2 T}{2b \gamma}.
\end{align}
We now have from Corollary \ref{cor:page_sum_online}
\begin{align*}
    d(\hat{x})^2 &\leq \frac{1}{T} \sum_{t=0}^{T-1} \mathbb{E} \left[\operatorname{dist}^2\left(\nabla H\left(x^{t+1}\right)+\nabla r_2\left(x^{t+1}\right), \partial\left(G+r_1\right)\left(x^{t+1}\right)\right) \right]  \\
    &\leq \frac{2\left(L+L_{r_2}\right)^2}{T}  \sum_{t=0}^{T-1} \mathbb{E}\left[\left\|x^{t+1}-x^t\right\|^2 \right]+ \frac{2}{T} \sum_{t=0}^{T-1}\mathbb{E}\left[\left\|\nabla H\left(x^{t}\right)-g^{t}\right\|^2\right] \\
    &\leq \frac{1}{T}\left( 2\left(L+L_{r_2}\right)^2 + \frac{2(1-p)L^2}{pb'}\right)  \sum_{t=0}^{T-1} \mathbb{E}\left[\left\|x^{t+1}-x^t\right\|^2 \right] + \frac{2}{T} \mathbf{1}_{\{b<N\}} \frac{T \sigma^2}{b} \\
     &\leq \frac{1}{T}\left( 2\left(L+L_{r_2}\right)^2 + \frac{2(1-p)L^2}{pb'}\right)  \frac{4}{\rho}\left(\Delta_0 + \frac{\sigma^2}{2pb\gamma}+ \mathbf{1}_{\{b<N\}} \frac{ \sigma^2 T}{2b \gamma}\right) +  \mathbf{1}_{\{b<N\}} \frac{2 \sigma^2}{b} \\
     & = \frac{8}{T \rho}\left( \left(L+L_{r_2}\right)^2 + \frac{(1-p)L^2}{pb'}\right)  \left(\Delta_0 + \frac{\sigma^2}{2pb\gamma}\right) +   \mathbf{1}_{\{b<N\}} \frac{4\sigma^2 }{\rho b \gamma}\left( \left(L+L_{r_2}\right)^2 + \frac{(1-p)L^2}{pb'}\right) + \mathbf{1}_{\{b<N\}} \frac{2 \sigma^2}{b} \\
     &=  \frac{8}{T \rho}\left( \left(L+L_{r_2}\right)^2 + \frac{(1-p)L^2}{pb'}\right)  \left(\Delta_0 + \frac{\sigma^2}{2pb\gamma}\right) + \mathbf{1}_{\{b<N\}} \frac{2 \sigma^2}{b} \left(\frac{2}{\rho \gamma} \left( \left(L+L_{r_2}\right)^2 + \frac{(1-p)L^2}{pb'}\right) + 1 \right) 
\end{align*}
Now with the choice $p =\frac{1}{\sqrt{b}}$, $b' = \sqrt{b}$, we have
\begin{align*}
    \left(L+L_{r_2}\right)^2 + \frac{(1-p)L^2}{pb'} &= \left(L+L_{r_2}\right)^2  + \frac{(\sqrt{b}-1)L^2}{\sqrt{b}} \\
    &\leq \left(L+L_{r_2}\right)^2 + L^2,
\end{align*}
which further implies
\begin{align*}
    d(\hat{x})^2 & \leq  \frac{8}{T \rho}\left( \left(L+L_{r_2}\right)^2 + L^2\right)  \left(\Delta_0 + \frac{\sigma^2}{2pb\gamma}\right) + \mathbf{1}_{\{b<N\}} \frac{2 \sigma^2}{b} \left(\frac{2}{\rho \gamma} \left( \left(L+L_{r_2}\right)^2 + L^2\right) + 1 \right) \\
    &= \frac{8C}{T \rho}  \left(\Delta_0 + \frac{\sigma^2}{2pb\gamma}\right) + \mathbf{1}_{\{b<N\}} \frac{2 \sigma^2}{b} \left(\frac{2C}{\rho \gamma} + 1 \right),
\end{align*}
where $C = \left(L+L_{r_2}\right)^2 + L^2$. Now we can pick $b = \lceil\sigma^2/(\alpha \varepsilon^2 )\rceil$, where $ \alpha = \frac{1}{4} \frac{\rho L}{4C  +\rho L}$, with the condition that $N > \lceil\sigma^2/(\alpha \varepsilon^2 )\rceil$. For arbitrary large $N$, this condition is naturally satisfied. If the condition does not satisfy, we go back to finite-sum setting where $b=N$. We have 
\begin{align*}
   d(\hat{x})^2 \leq  \frac{8C}{T \rho}  \left(\Delta_0 + \frac{\alpha \varepsilon^2}{2p\gamma}\right) + 2 \alpha \varepsilon^2 \left(\frac{2C}{\rho \gamma} + 1 \right).
\end{align*}
From the choice $p = 1/b'$, we note that
\begin{align*}
    \gamma = L \sqrt{\frac{1-p}{pb'}} = L \sqrt{\frac{b'-1}{b'}} \geq \frac{L}{2},
\end{align*}
since $b>1$. We then have
\begin{align*}
    \alpha = \frac{1}{4} \frac{\rho L}{4C  +\rho L} =\frac{1}{4} \frac{\rho L/2}{2C  +\rho L/2} \leq
    \frac{1}{4} \frac{\rho \gamma}{2C  +\rho\gamma}.
\end{align*}
Thus, we can choose 
\begin{align*}
    T = \frac{16C}{\varepsilon^2 \rho} \left(\Delta_0 + \frac{\alpha \varepsilon^2}{2p\gamma}\right) = \frac{16C\Delta_0}{\varepsilon^2 \rho} + \frac{8C\alpha}{\rho p \gamma} \leq \frac{16C\Delta_0}{\varepsilon^2 \rho} + \frac{2C\rho \gamma}{p \rho \gamma (2C + \rho\gamma)},
\end{align*}
in order to obtain that 
\begin{align*}
    d(\hat{x})^2 \leq \frac{\varepsilon^2}{2} + \frac{\varepsilon^2}{2} = \varepsilon^2.
\end{align*}
In other words, we need 
\begin{align*}
    T \leq \frac{16C\Delta_0}{\varepsilon^2 \rho} + \frac{2C\rho \gamma}{p \rho \gamma (2C + \rho\gamma)} 
\end{align*}
iterations to obtain $d(\hat{x}) \leq \varepsilon$. Finally, we obtain the number of gradient computation as follows
\begin{align*}
    \# \operatorname{grad} &= b + T(pb + (1-p)b') \\
    &\leq b + \left(\frac{16C\Delta_0}{\varepsilon^2 \rho} + \frac{2C\rho \gamma \sqrt{b}}{\rho \gamma (2C + \rho\gamma)} \right) \left(\sqrt{b} + \left( 1 - \frac{1}{\sqrt{b}} \right) \sqrt{b} \right) \\
    &\leq b + \left(\frac{16C\Delta_0}{\varepsilon^2 \rho} + \frac{2C\rho \gamma \sqrt{b}}{\rho \gamma (2C + \rho\gamma)} \right) 2 \sqrt{b} \\
    &= b \left(1  +  \frac{4C\rho \gamma}{\rho \gamma (2C + \rho\gamma)}\right)  + \frac{32C\Delta_0\sqrt{b}}{\varepsilon^2 \rho} \\
    &= \mathcal{O} \left(b + \frac{\sqrt{b}}{\varepsilon^2} \right).
\end{align*}

\section{Proof of Lemma \ref{lem:template} with Gap Function}
We first change the proof for the template inequality Lemma \ref{lem:template} with gap function and under the milder setting (without Assumption \ref{ass:r_2_smooth}). 
Since $H$ and $r_2$ are $\rho_{r_2}$-convex and $\rho_H$-convex, we have that
\begin{align*}
r_2\left(x^{t+1}\right) &\geq r_2\left(x^t\right)+\left\langle w_t, x^{t+1}-x^{t}\right\rangle+\rho_{r_2}\left\|x^{t+1}-x^{t}\right\|^2 / 2,\\
H\left(x^{t+1}\right) &\geq H\left(x^t\right)+\left\langle\nabla H\left(x^t\right), x^{t+1}-x^{t}\right\rangle+\rho_H\left\|x^{t+1}-x^{t}\right\|^2/2.
\end{align*}
By the strong convexity of $G+r_1$ and the fact that $g^t+w^t$ is a subgradient of $(G+r_1)(x^{t+1})$, we have 
\begin{align*}
(G+r_1)\left(x^t\right) &\leq (G+r_1)\left(x^{t+1}\right)+\left\langle g^t + w^t, x^t-x^{t+1}\right\rangle+\rho_{G+r_1}\left\|x^{t} - x^{t+1}\right\|^2/2.
\end{align*}
Adding the above inequalities, we obtain
\begin{align*}
F\left(x^{t+1}\right) &\leq F\left(x^t\right)+\left\langle x^{t+1}-x^t, g^t -\nabla H\left(x^t\right)\right\rangle -\rho \left\|x^{t+1} - x^{t} \right\|^2/2,
\end{align*}
where $\rho=\rho_H + \rho_{r_2} +\rho_{G + r_1}$. Since $\inner{a}{b} \leq \frac{1}{2\gamma} \|a\|^2 + \frac{\gamma}{2}\|b\|^2$, we obtain the desired template inequality. \qed
\section{Proof of Lemma \ref{lem:key}}
We have that 
\begin{align*}
    \operatorname{gap} \left(x^{t+1}\right) &= \max _{x \in \mathbb{R}^n}G\left(x^{t+1}\right)-G(x)-\left\langle\nabla H\left(x^{t}\right) + w^{t}, x^{t+1}-x\right\rangle \\
    &= \max _{x \in \mathbb{R}^n} G\left(x^{t+1}\right)-G(x) - \left\langle g^t + w^{t}, x^{t+1}-x\right\rangle  + \left\langle g^t - \nabla H\left(x^{t}\right), x^{t+1}-x\right\rangle.
\end{align*}
Since that $G+r_1$ is $ \rho_{G+r_1}$-strongly-convex, we have 
\begin{align*}
    -\rho_{G+r_1} \|x^{t+1}-x\|^2/2 &\geq (G+r_1)(x^{t+1}) - (G+r_1)(x) - \inner{g^t + w^t}{x^{t+1} -x}. 
\end{align*}
Taking expectation and summing from $0$ to $T-1$, we obtain
\begin{align}
     &\sum_{t=0}^{T-1} \mathbb{E} \left[G(x^{t+1}) - G(x) - \inner{g^t +  w^t}{x^{t+1} -x} \right] \leq - \frac{\rho_G}{2} \sum_{t=0}^{T-1} \mathbb{E} \left[\|x^{t+1}-x\|^2\right].
     \label{eq:diff_bound}
\end{align}
On the other hand, we have that $\left\langle g^t - \nabla H\left(x^{t}\right), x^{t+1}-x\right\rangle \leq \frac{1}{2 \eta} \|g^t - \nabla H\left(x^{t}\right) \|^2  + \frac{\eta}{2} \| x^{t+1}-x\|^2$, for $\eta \in \mathbb{R}$ to be chosen later. This further implies that
\begin{align}
      \sum_{t=0}^{T-1}  \mathbb{E} \left[\left\langle g^t - \nabla H\left(x^{t}\right), x^{t+1}-x\right\rangle \right] &\leq \frac{1}{2 \eta}  \sum_{t=1}^{T}  \mathbb{E} \left[\|g^t - \nabla H\left(x^{t}\right) \|^2\right]  + \frac{\eta}{2}  \sum_{t=0}^{T-1} \mathbb{E} \left[\| x^{t+1}-x\|^2 \right] \notag\\
     &\leq \frac{\left(1-p\right) L^2}{2 \eta p b^{\prime}}\sum_{t=1}^{T} \mathbb{E}\left[\left\|x^{t+1}-x^t\right\|^2 \right] + \frac{\eta}{2}  \sum_{t=1}^{T} \mathbb{E} \left[\| x^{t+1}-x\|^2 \right],
     \label{eq:inner_prod_bound}
\end{align}
where the second inequality follows from Corollary \ref{cor:page_sum}.
From (\ref{eq:diff_bound}) and (\ref{eq:inner_prod_bound}), we choose $\eta = \rho_{G+r_1}$ to obtain
\begin{align*}
    \frac{1}{T}\sum_{t=0}^{T-1}  \mathbb{E} \left[\operatorname{gap} \left(x^{t}\right) \right] &=\frac{1}{T}\sum_{t=0}^{T-1}  \mathbb{E} \left[G(x^{t}) - G(x) - \inner{g^t}{x^{t} -x} \right]+ \frac{1}{T}\sum_{t=0}^{T-1}\mathbb{E}\left[ \left\langle g^t - \nabla H\left(x^{t}\right), x^{t}-x\right\rangle \right] \\
    &\leq \frac{\left(1-p\right) L^2}{2 \eta p b^{\prime} T}\sum_{t=0}^{T-1} \mathbb{E}\left[\left\|x^{t+1}-x^t\right\|^2 \right].
\end{align*}
Since $\operatorname{gap}(\hat{x}) \leq \frac{1}{T}\sum_{t=1}^{T}  \mathbb{E} \left[\operatorname{gap} \left(x^{t}\right) \right]$, the claim holds. \qed
\section{Proof of Theorem \ref{them:dca_page_gap_fin_sum}}
    Since the template inequality still holds for this setting, we can proceed similar to the proof of Theorem 9 with $\gamma$'s choice and  $b'$'s choice in \eqref{eq:gamma_choice} and \eqref{eq:b'_choice}, respectively, to obtain the following inequality similar to \eqref{eq:iter_diff_sum}
    \begin{align*}
        \sum_{t=0}^{T-1}\mathbb{E}\left[ \left\|x^{t+1}-x^t\right\|^2 \right] \leq \frac{4\Delta_0}{\rho},
    \end{align*}
    where $\rho=\rho_H + \rho_{r_2} +\rho_{G + r_1}$ and $\Delta_0 = F\left(x^0\right)-F^*$. Combining this equality with Lemma \ref{lem:key}, we obtain 
    \begin{align*}
    \operatorname{gap}(\hat{x}) \leq \frac{2 \left(1-p\right) L^2 \Delta_0}{ \rho_{G+r_1} p b^{\prime} \rho  T}. 
\end{align*}
    Therefore, we need $T$ iterations as in \eqref{eq:refine_T_finite} for $\operatorname{gap}(\hat{x}) \leq \varepsilon$. 
    With $p = 1 / b'$, $b = N$ and $b'= \lceil\sqrt{b} \rceil$, we have the number of gradient computations is
\begin{align*}
    b + T(pb + (1-p)b') \leq N + \frac{2 \left(b'-1\right) L^2 \Delta_0}{ \rho_{G+r_1} b^{\prime} \rho \varepsilon} 2b'= \mathcal{O} \left(N + \sqrt{N}/\varepsilon\right). 
\end{align*}\qed
\section{Proof of Theorem \ref{them:dca_page_gap_onl}}
    Since the template inequality still holds for this setting, we can proceed similar to the proof of Theorem 12 with $\gamma$'s choice and  $b'$'s choice in \eqref{eq:gamma_choice} and \eqref{eq:b'_choice}, to obtain the below inequality similar to \eqref{eq:sum_diff_iters_online}
    \begin{align*}
        \sum_{t=0}^{T-1}\mathbb{E}\left[ \left\|x^{t+1}-x^t\right\|^2 \right] \leq \frac{4}{\rho} \left(\Delta_0 + \frac{\sigma^2}{2pb\gamma}+ \mathbf{1}_{\{b<N\}} \frac{ \sigma^2 T}{2b \gamma}\right),
    \end{align*}
    where $\rho=\rho_H + \rho_{r_2} +\rho_{G + r_1}$ and $\Delta_0 = F\left(x^0\right)-F^*$. Combining this equality with Lemma \ref{lem:key}, we obtain 
    \begin{align}
    \operatorname{gap}(\hat{x}) &\leq \frac{2 \left(1-p\right) L^2}{ \rho_{G+r_1} p b^{\prime} \rho  T} \left(\Delta_0 + \frac{\sigma^2}{2pb\gamma}+ \mathbf{1}_{\{b<N\}} \frac{ \sigma^2 T}{2b \gamma}\right) \notag \\
    &= \frac{2 \left(1-p\right) L^2}{ \rho_{G+r_1} p b^{\prime} \rho  T} \left(\Delta_0 + \frac{\sigma^2}{2pb\gamma}\right) + \mathbf{1}_{\{b<N\}}  \frac{ \left(1-p\right) L^2 \sigma^2}{ \rho_{G+r_1} p b^{\prime} \rho b \gamma}.
    \label{eq:gap_bound_refined_online}
    \end{align}
     We then choose minibatch size $b = \lceil\sigma^2/(\alpha \varepsilon)\rceil$, where 
     \begin{align}
         \alpha = \frac{\rho_{G+r_1} \rho \gamma p}{2L^2} \leq \frac{\rho_{G+r_1} \rho \gamma p b'}{2 L^2 (b'-1)},
         \label{eq:alp_bound}
     \end{align}
     We need the condition $N > \lceil\sigma^2/(\alpha \varepsilon^2 )\rceil$ for this choice of $b$. For arbitrarily large $N$, this condition is naturally satisfied. If the condition does not satisfy, we can go back to finite-sum setting where $b=N$.

    As $\sigma^2 / b < \alpha \varepsilon$, we can bound the RHS of \eqref{eq:gap_bound_refined_online} as
    \begin{align*}
        \operatorname{RHS} &\leq \frac{2 \left(1-p\right) L^2}{ \rho_{G+r_1} p b^{\prime} \rho  T} \left(\Delta_0 + \frac{\alpha \varepsilon}{2p\gamma}\right) + \frac{ \left(1-p\right) L^2 \alpha \varepsilon}{ \rho_{G+r_1} p b^{\prime} \rho \gamma} \\
        &=  \frac{2 \left(1-p\right) L^2}{ \rho_{G+r_1} p b^{\prime} \rho  T} \left(\Delta_0 + \frac{\alpha \varepsilon}{2p\gamma}\right) + \frac{ \left(b' - 1\right) L^2 \alpha \varepsilon}{ \rho_{G+r_1} b^{\prime} \rho \gamma} \\
        &\leq \frac{2 \left(b'-1\right) L^2}{ \rho_{G+r_1} b^{\prime} \rho  T} \left(\Delta_0 + \frac{\alpha \varepsilon}{2p\gamma}\right) + \frac{\varepsilon}{2},
    \end{align*}
    where the second equality follows from the choice $p = 1/b'$ and the last inequality follows from \eqref{eq:alp_bound}.
    
    Therefore, by setting the first term to $\varepsilon/2$, we need $T$
    iterations as in \eqref{eq:refine_T_online} for $\operatorname{gap}(\hat{x}) \leq \varepsilon$. With $b'= \lceil\sqrt{b} \rceil$, we have the number of gradient computations is
\begin{align*}
    b + T(pb + (1-p)b') &\leq b + \frac{4 \left(b'-1\right) L^2}{ \rho_{G+r_1} b^{\prime} \rho \varepsilon} \left(\Delta_0 + \frac{\alpha \varepsilon}{2p\gamma}\right) 2b'\\
    &\leq b + \frac{8 (\sqrt{b}-1) L^2 \Delta_0}{ \rho_{G+r_1} \rho \varepsilon} + \frac{4 \alpha L^2 b}{ \rho_{G+r_1} \rho}\\
    &= \mathcal{O} \left(b + \sqrt{b}/\varepsilon\right). 
\end{align*} 
\qed

\end{document}